\documentclass[12pt,psamsfonts]{amsart}

\usepackage{amssymb,amsfonts,amsmath}
\usepackage{enumerate}
\usepackage{mathrsfs}
\usepackage{fullpage}
\usepackage{xspace}
\usepackage[margin=1.0in]{geometry}
\usepackage{tcolorbox}
\usepackage{tikz-cd}
\usepackage{color}
\usepackage{aliascnt}
\usepackage[foot]{amsaddr}
\usepackage{hyperref}
\usepackage{graphicx}
\usepackage[algoruled,linesnumbered]{algorithm2e}

\usepackage{enumitem}
\setlist[enumerate]{label=$(\mathrm{\arabic*})$, leftmargin=*}
\setlist[itemize]{leftmargin=*}

\newtheorem{thm}{Theorem}[section]

\newaliascnt{theo}{thm}

\aliascntresetthe{theo}

\newaliascnt{cor}{thm}
\newtheorem{cor}[cor]{Corollary}
\aliascntresetthe{cor}

\newaliascnt{prop}{thm}
\newtheorem{prop}[prop]{Proposition}
\aliascntresetthe{prop}

\newaliascnt{lem}{thm}
\newtheorem{lem}[lem]{Lemma}
\aliascntresetthe{lem}

\newaliascnt{conj}{thm}
\newtheorem{conj}[conj]{Conjecture}
\aliascntresetthe{conj}

\newaliascnt{que}{thm}

\aliascntresetthe{que}

\newaliascnt{ass}{thm}

\aliascntresetthe{ass}

\newaliascnt{defnot}{thm}

\aliascntresetthe{defnot}

\newaliascnt{princ}{thm}

\aliascntresetthe{princ}

\theoremstyle{remark}
\newaliascnt{rem}{thm}
\newtheorem{rem}[rem]{Remark}
\aliascntresetthe{rem}

\theoremstyle{definition}

\newaliascnt{defn}{thm}
\newtheorem{defn}[defn]{Definition}
\aliascntresetthe{defn}

\newaliascnt{exmp}{thm}

\aliascntresetthe{exmp}

\newaliascnt{notn}{thm}
\newtheorem{notn}[notn]{Notation}
\aliascntresetthe{notn}

\newtheorem{conv}[thm]{Convention}

\newcommand{\Z}{\mathbb{Z}\xspace}

\newcommand{\Q}{\mathbb{Q}\xspace}
\newcommand{\G}{\mathbb{G}\xspace}
\DeclareMathOperator{\Spec}{Spec}
\DeclareMathOperator{\res}{res}

\DeclareMathOperator{\ord}{ord}

\DeclareMathOperator{\alb}{alb}
\DeclareMathOperator{\img}{Im}

\DeclareMathOperator{\Pic}{Pic}

\DeclareMathOperator{\Fil}{Fil}

\DeclareMathOperator{\Alb}{Alb}
\DeclareMathOperator{\Symb}{Symb}

\DeclareMathOperator{\CH}{CH}
\DeclareMathOperator{\rk}{rank}

\makeatletter
\let\c@equation\c@thm
\makeatother
\numberwithin{equation}{section}

\newcommand{\red}{\color{red}}

\newcommand{\black}{\color{black}}

\newcommand{\new}{_{\mathrm{comp}}}

\title{Torsion phenomena for zero-cycles on a product of curves over a number field}

\author[*]{Evangelia Gazaki*} \address[*]{\normalfont Department of Mathematics, University of Virginia, 221 Kerchof Hall, 141 Cabell Dr., Charlottesville, VA, 22904, USA. Email: \texttt{eg4va@virginia.edu}}
\author[**]{Jonathan Love**} \address[**]{\normalfont Department of Mathematics and Statistics, Burnside Hall, McGill University, 805 Sherbrooke Street West, Montreal, Quebec H3A 0B9. Email: \texttt{jon.love@mcgill.ca}}
\begin{document}

	\begin{abstract}
		For a smooth projective variety $X$ over an algebraic number field $k$ a conjecture of Bloch and Beilinson predicts that the kernel of the Albanese map of $X$ is a torsion group. In this article we consider a product $X=C_1\times\cdots\times C_d$ of smooth projective curves and show that if the conjecture is true for any subproduct of two curves, then it is true for $X$. For a product $X=C_1\times C_2$ of two curves over $\Q$ with positive genus we construct many nontrivial examples that satisfy the weaker property that the image of the natural map $J_1(\Q)\otimes J_2(\Q)\xrightarrow{\varepsilon}\CH_0(C_1\times C_2)$ is finite, where $J_i$ is the Jacobian variety of $C_i$. Our constructions include many new examples of non-isogenous pairs of elliptic curves $E_1, E_2$ with positive rank, including the first known examples of rank greater than $1$. Combining these constructions with our previous result, we obtain infinitely many nontrivial products  $X=C_1\times\cdots\times C_d$ for which the analogous map $\varepsilon$ has finite image.
	\end{abstract}

	\maketitle


	\section{Introduction} 
	Let $X$ be a smooth projective variety over an algebraic number field $k$. Consider the Chow group $\CH_0(X)$ of zero-cycles modulo rational equivalence on $X$ and denote by $F^1(X)$ the subgroup of cycles of degree zero. Let $\Alb_X$ denote the Albanese variety of $X$, and $F^2(X)$ the kernel of the Albanese map
	\[\alb_X:F^1(X)\to \Alb_X(k).\]
 Set $X_{\overline{\Q}}:=X\otimes_k\overline{\Q}$. The following is a famous conjecture of Beilinson (\cite{Beilinson1984}). 
\begin{conj} The Albanese map $F^1(X_{\overline{\Q}})\rightarrow\Alb_{X_{\overline{\Q}}}(\overline{\Q})$ is injective. That is, $F^2(X_{\overline{\Q}})=0$. 
\end{conj} 
 When $C$ is a curve, the Albanese map coincides with the isomorphism between $\Pic^0(C)$ and the $k$-rational points of the Jacobian $J_C$ of $C$, and hence $F^2(C)=0$ in this case. In higher dimensions however there is hardly any evidence for this conjecture. 

Next suppose that $k$ is an arbitrary algebraic number field. Using standard push-forward and pull-back arguments for Chow groups, the above conjecture implies the following. 
\black 	
	\begin{conj}\label{bbconj}
		The group $F^2(X)$ is torsion.
	\end{conj} 
\black 	  The goal of this article is to obtain some weak evidence for \autoref{bbconj} \black for a product $X=C_1\times\cdots\times C_d$ of smooth projective curves over $k$. 

	\subsection*{A product formula for Somekawa $K$-groups and applications}
	
	Let $J_i$ be the Jacobian variety of $C_i$. Under the assumption $X(k)\neq\emptyset$,  Raskind and Spiess (\cite[(1.4)]{Raskind/Spiess2000}) established  an isomorphism
\[\CH_0(X)\simeq\Z\oplus\bigoplus_{\nu=1}^d\bigoplus_{1\leq i_1<\cdots<i_\nu\leq d} K(k;J_{i_1},\ldots,J_{i_\nu}),\]	which in turn yields an isomorphism (see \cite[p.~4]{Yamazaki2005})
	 \[F^2(X)\simeq\bigoplus_{\nu=2}^d\bigoplus_{1\leq i_1<\cdots<i_\nu\leq d} K(k;J_{i_1},\ldots,J_{i_\nu}),\]  where $K(k;J_{i_1},\ldots,J_{i_\nu})$ is the Somekawa $K$-group attached to $J_{i_1},\ldots, J_{i_\nu}$ (cf.~\cite{Somekawa1990}). 
	For semi-abelian varieties $G_1,\ldots, G_r$ over $k$, the $K$-group $K(G_1,\ldots, G_r)$ is a generalization of the Milnor $K$-group $K_r^M(k)$ of the field $k$.  In fact when $G_1=\cdots=G_r=\G_m$ there is an isomorphism $K(\G_m,\cdots,\G_m)\simeq K_r^M(k)$. For the latter, there is a well-defined map 
	 \black
	\[K_r^M(k)\otimes K_s^M(k)\to K_{r+s}^M(k)\] for $r,s\geq 0$, given by concatenation of symbols. In \autoref{productformulasection} we prove the following  analog for Somekawa $K$-groups attached to abelian varieties. 
\begin{prop}\label{mainpropintro} (cf.~\autoref{mainprop}) Let $A_1,\ldots, A_r$ be abelian varieties over a perfect field $k$ with $r\geq 3$. Let $L/k$ be a finite extension. There is a well-defined homomorphism 
			\[ \Phi_L:K(k;A_1,A_2)\otimes A_3(L)\otimes\cdots\otimes A_r(L)\to K(k;A_1,\ldots, A_r).\]
		\end{prop}
	 \noindent	In the simplest case of an element $\{a_1,a_2\}_{k/k}\otimes a_3\otimes\cdots\otimes a_r$ with $a_i\in A_i(k)$ for all $i$, the map $\Phi_L$ is simply given by concatenation, 
		\[\Phi_L(\{a_1,a_2\}_{k/k}\otimes a_3\otimes\cdots\otimes a_r)=\{a_1,\ldots,a_r\}_{k/k}.\]  The general definition is more involved; we omit it here for the purposes of the introduction.  We will refer to the homomorphism $\Phi$ as a \textit{product formula for Somekawa $K$-groups attached to abelian varieties}. \black While the proof of \autoref{mainpropintro} is relatively straightforward, it yields the following important corollary, allowing us to deduce \autoref{bbconj} for a product of arbitrarily many curves, assuming it is known for each subproduct of two curves.\black
		\begin{cor} \black  (cf.~ \autoref{main}) \label{mainintro} Let $d\geq 2$ and $X=C_1\times\cdots\times C_d$ be a product of smooth projective curves over a number field $k$ with $C_i(k)\neq\emptyset$ for $i\in\{1,\ldots,d\}$. Suppose that for each $1\leq i<j\leq d$ the group $F^2(C_i\times C_j)$ is torsion. Then $F^2(X)$ is torsion.
	\end{cor} 
			
	
%
%
\black 
\subsection*{The componentwise subgroup}	
 
	Proving \autoref{bbconj} 
	for even a single product $X=C_1\times C_2$ of two non-rational  curves $C_1,C_2$ seems to be far out of reach. 
	 To date, the best results towards understanding the structure of $F^2(X)$ are due to  
	Langer and Saito (\cite{Langer/Saito1996}) and Langer and Raskind (\cite{Langer/Raskind1996}), who proved that for the self-product $X=E\times E$ of an elliptic curve $E$ over $\Q$ without and with potential CM respectively, the $p$-primary torsion subgroup of $\CH_0(X)$ is finite for every prime $p$ coprime to the conductor $N$ of $E$. 
	
	 In this article we give evidence for a weaker question. We consider the homomorphism 
	\[\varepsilon_X:\CH_0(C_1)\otimes\cdots\otimes \CH_0(C_d)\to\CH_0(X)\] induced by the intersection product. We define \textit{the componentwise subgroup} $F^2(X)\new$ of $F^2(X)$ to be $F^2(X)\new:=F^2(X)\cap\img(\varepsilon_X)$ (see \autoref{compwise}). We propose the following conjecture, which is automatically implied by \autoref{bbconj}.  
	\black 
	\begin{conj}\label{bbcompconj}
		The subgroup $F^2(X)\new$ is finite.
	\end{conj} 
	 
To the authors' knowledge this article is the first instance where \autoref{bbcompconj} is stated at this level of generality. In the special case when $X=C_1\times C_2$ is a product of two curves with $X(k)\neq\emptyset$, the group $F^2(X)\new$ is precisely the image of the natural map 
	\[\varepsilon_X: J_1(k)\otimes J_2(k)\to F^2(X).\]  When both $C_i$ are elliptic curves, \black  this conjecture has been considered in \cite{Love2020} and \cite{Prasanna/Srinivas2022}, both of which provide new examples of such surfaces $X$ for which $F^2(X)\new$ is finite. Note that if  the base change \black  $X_L$ satisfies \autoref{bbcompconj} for all finite extensions $L/k$, then $X$ satisfies \autoref{bbconj} (see \autoref{compconjtofullconj}), so \autoref{bbcompconj} can be seen as a first step towards \autoref{bbconj} for products of curves. \black
	
	
	\subsection*{Detecting torsion classes for a product of two curves}	
	In \autoref{finiteoverk} we construct many nontrivial examples  for which  
	\autoref{bbcompconj} can be verified. For the sake of introduction, we choose to state our results in their simplest form. 	
	 Assume  the curves $C_1$ and $C_2$ each have a distinguished rational point $O_i\in C_i(k)$. Then every $k$-rational point $(P,Q)$ on $X=C_1\times C_2$ gives rise to a zero-cycle in $F^2(X)\new$, namely
	\[z_{P,Q}:=[P,Q]-[P,O_2]-[O_1,Q]-[O_1,O_2]=\pi_1^\star([P]-[O_1])\cdot\pi_2^\star([Q]-[O_2]),\] 
	 where $\pi_i:X\to C_i$ is the natural projection.  When $X$ is a product of two elliptic curves, such zero-cycles in fact generate $F^2(X)\new$, but this is generally not true for higher genus curves. \black  The following key proposition can be used to show that certain cycles of this form are torsion. 
	
	\begin{prop} (cf.~ \autoref{torsionfromhyper})  \label{torsionfromhyperintro}
		Let $H$ be an elliptic  (resp.~ hyperelliptic) curve, let $P\in H(k)$, and let $\phi_i:H\to C_i$ be a regular map for each $i=1,2$. Suppose there exists $W\in H(\overline{k})$, fixed by negation (resp.~ the hyperelliptic involution), such that $[\phi_i(W)]-[O_i]\in J_i(\overline{k})$ is torsion for $i=1,2$. Then the zero-cycle $z_{\phi_1(P),\phi_2(P)}$
		is torsion.
	\end{prop} 
	
	
	This proposition is the main new tool we use to construct torsion classes in $F^2(X)\new$. While the result is a relatively direct consequence of Weil reciprocity of the Somekawa $K$-group $K(k;J_1,J_2)$,  its power lies in the fact that its conditions are quite easy to satisfy, allowing us to apply it to large classes of examples. \black
	
 We first observe two direct consequences of \autoref{torsionfromhyperintro} which are likely known to the experts, but we could not find them in the literature at this level of generality. First, any product $X=E_1\times E_2$ of isogenous rank $1$ elliptic curves over a number field satisfies \autoref{bbcompconj}. Second, let $E$ be an elliptic curve over $\Q$ with potential CM by the ring of integers of a quadratic imaginary field $K$. Suppose that $E(\Q)$ has rank $1$ and $E(K)$ has rank $2$. Then \autoref{bbcompconj} holds for the base change $E_K\times E_K$ to $K$ (cf.~ \autoref{CMcase}).
	
	The following Corollary gives a less trivial application of \autoref{torsionfromhyperintro}. 
	\begin{cor}\label{hyperrank1intro} (cf.~ \autoref{hyperselfproduct}) Let $H$ be a hyperelliptic curve over $k$ with Jacobian $J$. Suppose $\rk J(k)=1$, and that there exists $P,W\in H(k)$, with $W$ fixed by the hyperelliptic involution, and $[P]-[W]\in J(k)$ is of infinite order. Then $H\times H$ satisfies \autoref{bbcompconj}. 
	\end{cor} 
	\black 
	There are $860$ genus $2$ curves over $\Q$ in the LMFDB~\cite{lmfdb} with conductor at most $10000$ satisfying the criteria in \autoref{hyperrank1intro}, assuming that analytic rank equals algebraic rank for each of the curves considered (as predicted by the Birch and Swinnerton-Dyer conjecture). To our knowledge these are the first examples of \autoref{bbcompconj} (and the first pieces of evidence towards \autoref{bbconj}) involving products of higher genus curves of positive rank. 

	The most substantial contribution of \autoref{torsionfromhyperintro} is for products $X=E_1\times E_2$ of elliptic curves with positive rank over $\Q$ and fully rational $2$-torsion, a situation which is discussed in \autoref{newconstruction}. \black  For any such pair of elliptic curves, a theorem of Scholten (\cite[Theorem 1]{Scholten}) provides an explicit hyperelliptic curve $H$ of genus $2$ over $k$ with  explicit regular maps $H\xrightarrow{\phi_i}E_i$ for $i=1,2$.\footnote{Constructions of such a curve $H$ were known prior to Scholten, for instance in \cite{freykani1991}, but we use Scholten's construction because it provides explicit equations.}  We thus obtain the following important corollary. 
	
	\begin{cor}\label{2torsionexmps} In the above set-up suppose that the set $\{z_{\phi_1(P),\phi_2(P)}:P\in H(k)\}$ generates a finite index subgroup of $F^2(E_1\times E_2)\new$. Then $E_1\times E_2$ satisfies \autoref{bbcompconj}.
	\end{cor}
 We present an algorithm (cf.~ \autoref{tensoralg}) that can be used to check the above criteria computationally. Applying this algorithm to a collection of curves from the LMFDB~\cite{lmfdb} allowed us to find many explicit examples.  This includes the first known examples where $F^2(X)\new$ is finite with $E_1,E_2$ \black  non-isogenous curves with rank greater than $1$. Specifically, for products $X=E_1\times E_2$ with $E_1, E_2$  non-isomorphic elliptic curves over $\Q$ with torsion subgroup $\Z/2\Z\oplus\Z/2\Z$,  we verify that $F^2(X)\new$ is finite \black for $2602$ pairs of rank $1$ curves (out of $4950$ pairs tested), $995$ pairs of rank $2$ curves (out of  $4950$), $17$ pairs of rank $3$ curves (out of $190$), \black $3311$ pairs with ranks $1$ and $2$ respectively (out of $10000$), $955$ pairs with ranks $1$ and $3$ respectively (out of $10000$), and $615$ pairs with ranks $2$ and $3$ respectively (out of $10000$). We also found several 
	self-products $E\times E$   satisfying \autoref{bbcompconj} with $E$  of rank $2$ or $3$. We refer to \autoref{data} for more details.

	We note that to our knowledge there are only two previously known constructions of non-isogenous rank $1$ pairs with product satisfying \autoref{bbcompconj}. The first one is due to Prasanna and Srinivas, who in an upcoming preprint\footnote{ The authors first learned of this result by private communication in 2018.} \cite{Prasanna/Srinivas2022} constructed two such pairs: one pair of two non-isogenous curves with conductor $37$, and the other pair with two non-isogenous curves of conductor $91$. \black The second construction is due to the second author, who constructed a $2$-parameter family $E_{s,t}$ of elliptic curves over a number field $k$ with the property that for each $s\in k$, and all $t_1,t_2\in k$ outside a finite subset depending on $s$, if $E_{s,t_1}(k)$ and $E_{s,t_2}(k)$ both have rank $1$, then $E_{s,t_1}\times E_{s,t_2}$ satisfies \autoref{bbcompconj}. In \autoref{previousconstructions} we recall some details of these constructions.

	\subsection*{Products of more curves} The proof of \autoref{mainintro} yields the following important corollary,  which allows us to produce infinitely many nontrivial products for which \autoref{bbcompconj} is true. 
	\begin{cor} Let $E_1, E_2$  be any of the rank $1$ pairs  mentioned above such that $F^2(E_1\times E_2)\new$ is finite. Let $d\geq 2$ and $X=C_1\times\cdots\times C_d$ with each $C_i$ either an elliptic curve isogenous to $E_1$ or $E_2$, or a curve with $C_i(\Q)\neq\emptyset$ and $\rk(J_{C_i}(\Q))=0$. Then $F^2(X)\new$ is finite. 
\end{cor}  The above corollary is one of multiple corollaries we obtain by combining \autoref{mainintro} and the various constructions that give finiteness of $F^2(C_1\times C_2)\new$.

		\subsection*{Acknowledgments} The first author was partially supported by the NSF grants DMS-~2001605 and DMS-~2302196, and the second author by a CRM-ISM postdoctoral fellowship. We are very thankful to 
		Professors Kartik Prasanna and Akshay Venkatesh for introducing us to this question and to Kartik Prasanna and Vasudevan Srinivas for allowing us to include in our article something about their  work in preparation. We are also thankful to Professors Spencer Bloch and Shuji Saito for showing interest in our article. Lastly we would like to thank our referees whose suggestions helped improve this article. 
		\black 
		
		\vspace{5pt}
		\section{Background} 
		
		 Throughout this article, $k$ will be a perfect field. 
		
		\begin{notn} Let $X$ be a smooth projective variety over $k$. We will always assume that $X$ has a $k$-rational point. \black  By a point $x\in X$ we will always mean a closed point and we will denote by $k(x)$ the residue field of $x$. \black Given a field extension $L/k$, we let $X(L)$ denote the set of $L$-rational points of $X$ with respect to the extension $k\hookrightarrow L$, and $X_L$ denote the base change to $L$. We use parentheses, as in $(X_1)_L$, if the variety is labelled with an index.
			
			
			
			Given $P\in X(k)$ and a finite extension $L/k$, we also use $P$ to denote the $L$-rational point of $X_{L}$ induced by base change. We write $[P]_L$ to denote the divisor on $X_{L}$ supported at $P$ with multiplicity one; if it is clear from context that we are considering divisors on $X_{L}$, we may drop the subscript and just write $[P]$.

		\end{notn}
		\subsection{Zero-cycles}\label{zerocycles}
		Let $X$ be a smooth projective variety over $k$.
		  The group $\CH_0(X)$ is the quotient of the free abelian group $\displaystyle Z_0(X):=\bigoplus_{x\in X}\Z$ on all closed points of $X$ modulo rational equivalence. The generators of $\CH_0(X)$ will be written as classes $[x]$ of closed points. There is a well-defined degree map 
		\[\deg: \CH_0(X)\to \Z, [x]\mapsto[k(x):k].\] We will denote by $F^1(X)$ the kernel of $\deg$. Moreover, there is a higher dimensional analog of the Abel-Jacobi map,
		\[\alb_X:F^1(X)\to\Alb_X(k),\] called the Albanese map of $X$, where $\Alb_X$ is the Albanese variety of $X$ (dual to the Picard variety $\Pic^0(X)$). When $X$ is a curve, $\alb_X$ is precisely the Abel-Jacobi map, and $\Alb_X=J_X$ is the usual Jacobian of $X$. We will denote by $F^2(X)$ the kernel of $\alb_X$. 
		
%
	From now on we suppose $X=C_1\times\cdots\times C_d$ is a product of smooth projective curves over $k$. In this case $\Alb_X=J_1\times\cdots\times J_d$, where $J_i$ is the Jacobian of $C_i$, for $i\in\{1,\ldots,d\}$.  Let $\pi_i:X\to C_i$ be the projection to the $i$-th component for $i\in\{1,\ldots,d\}$, which induces a pullback 
	 homomorphism, $\pi_i^{\star}:\CH_0(C_i)\to\CH_0(X)$.  We can define a zero-cycle on $X$ given the input of a zero-cycle on each component. Namely, we consider the homomorphism 
				\begin{align}\label{eXdef}
					\varepsilon_X:\CH_0(C_1)\otimes\cdots\otimes\CH_0(C_d)\to \CH_0(X)
				\end{align}
				defined as follows. Let $x_i\in C_i$ for each $i\in\{1,\ldots,d\}$. Define 
				\[\varepsilon_X([x_1]\otimes\cdots\otimes[x_d])=\pi_1^\star([x_1])\cdot\pi_2^\star([x_2])\cdot\cdots\cdot \pi_d^\star([x_d]),\] and extend linearly. Here $\cdot$ is the intersection product. 
		\begin{defn}\label{compwise} 
				
We define the componentwise subgroup $\CH_0(X)\new$ of $\CH_0(X)$ to be the image of $\varepsilon_X$, and $F^2(X)\new:=\CH_0(X)\new\cap F^2(X)$.
		\end{defn}
		
		\black 
		
		%

		
		 It is not known whether $\varepsilon_X$ is surjective  onto $\CH_0(X)$, but we can obtain a surjective map by considering the base change of $X$ to  the residue fields of all closed points $x$ of $X$. Given $x\in X$, denote by $\varepsilon_{x}$ the composition
		\[\varepsilon_{x}:\CH_0((C_1)_{k(x)})\otimes\cdots\otimes \CH_0((C_d)_{k(x)})\xrightarrow{\varepsilon_{X_{k(x)}}}\CH_0(X_{k(x)})\xrightarrow{ (\pi_x)_\star}\CH_0(X),\] where $(\pi_x)_\star$   is the proper push-forward induced by the base change $X_{k(x)}\to X$. Then
		\begin{align}\label{CHsurj}
			\bigoplus_{x\in X}\CH_0((C_1)_{k(x)})\otimes\cdots\otimes \CH_0((C_d)_{k(x)})\xrightarrow{\oplus\varepsilon_{x}}\CH_0(X)
		\end{align}
		is a surjection.
		For, every closed point $x\in X$ has a $k(x)$-rational point $(x_1,\ldots,x_d)\in C_1(k(x))\times\cdots \times C_d(k(x))$ lying above it, so that $\varepsilon_x([x_1]\otimes \cdots\otimes [x_d])=[x]$.
		
	From now on assume that for each $i=1,\ldots,d$ there exists a fixed $k$-rational point $O_i\in C_i(k)$, so that we have an isomorphism \[\alb_{C_i}:F^1(C_1)=\Pic^0(C_i)\xrightarrow{\simeq} J_i(k),\] and a decomposition $\CH_0(C_i)\xrightarrow{\simeq}\Z\oplus J_i(k)$, sending the class $[x_i]$ of a closed point to the pair  $([k(x_i):k],[x_i]- [k(x_i):k][O_i])$. \black We get a similar decomposition, $\CH_0((C_i)_L)\simeq\Z\oplus J_i(L)$, over any finite extension $L/k$ by using the pullback $[O_i]_L$ of the zero-cycle $[O_i]$. By expanding the tensor products in \eqref{CHsurj}, we obtain a surjection
		\begin{align}\label{eXexpanded}
			\Z\oplus\bigoplus_{\nu=1}^d\bigoplus_{1\leq i_1<\cdots<i_\nu\leq d}  \left(\bigoplus_{L/k\text{ finite}} J_{i_1}(k(x))\otimes\cdots\otimes  J_{i_\nu}(k(x))\right)  \twoheadrightarrow\CH_0(X).
		\end{align}

		\black 
		 Raskind and Spiess determined the kernel of this map.
		 The relations are given by \emph{Somekawa $K$-groups} $K(k;J_{i_1},\ldots,J_{i_\nu})$. Before we recall their precise result, we review some necessary facts about Somekawa $K$-groups. 

		\black

		\vspace{1pt}
		\subsection{The Somekawa $K$-group} 
		
		 Let $G_1,\ldots, G_r$ be semi-abelian varieties over $k$. The Somekawa $K$-group $K(k;G_1,\ldots,G_r)$ attached to $G_1,\ldots, G_r$ was defined by K. Kato and Somekawa in \cite{Somekawa1990}, and it is a generalization of the Milnor $K$-group $K_r^M(k)$. 
		We note that the signs need to be dropped from the original definition (see \cite[p.~10]{Raskind/Spiess2000} for the correction). In what follows we will only need the explicit definition of this $K$-group for abelian varieties, so we omit the more general cases. \black 
		\begin{notn}\label{normresnotation}
			  Let $A$ be an abelian variety over $k$. Given a tower $F/L/k$ of field extensions,  we denote by $\res_{F/L}:A(L)\to A(F)$ and $N_{F/L}:A(F)\to A(L)$ the usual restriction and norm map for abelian varieties. Here for $N_{F/L}$ to be well-defined we need to assume $F/L$ is a finite extension. \black 
			 	 If we want to draw attention to the choice of extension $\phi:L\to F$ (for instance, if we are considering multiple embeddings of $L$ into $F$), we may use the notation $\res_{\phi}$ or $N_{\phi}$.   The restriction map is in fact defined for an arbitrary field extension $F/L$; it is the obvious embedding induced by precomposing with the morphism $\Spec F\to \Spec L$. The norm map is only defined for finite extensions. We note that when $F/L$ is a finite Galois extension with Galois group $G$, then $N_{F/L}(a)=\sum_{g\in G}ga$, for $a\in A(F)$. \black 
		\end{notn}
		 With these norm and restriction maps, an abelian variety $A$ becomes a \textit{Mackey functor} over $k$ (cf. \cite[p.~13]{Raskind/Spiess2000}). The following definition is due to Kahn. 
		 
		\begin{defn} (see \cite{Kahn1992}) Let $A_1,\ldots,A_r$ be abelian varieties over $k$. Their Mackey product $(A_1\otimes^M\cdots\otimes^M A_r)(k)$ is defined to be $\displaystyle\left(\bigoplus_{F/k\text{ finite}}A_1(F)\otimes \cdots\otimes A_r(F)\right)/(\mathbf{PF})$ where $(\mathbf{PF})$ is the subgroup generated by the following type of elements. 
			Let $F/L/k$ \black be a tower of finite extensions. Let $a_i\in A_i(F)$ for some $i\in\{1,\ldots,r\}$ and $a_j\in A_j(L)$ for every $j\neq i$. We require 
			\begin{equation}\label{PF}\left(a_1\otimes\cdots\otimes N_{F/L}(a_i)\otimes\cdots\otimes a_r\right)-\left(\res_{F/L}(a_1)\otimes\cdots\otimes a_i\otimes\cdots\otimes \res_{F/L}(a_r)\right)\in(\mathbf{PF}).
			\end{equation}
		\end{defn} The relation \eqref{PF} is known as \textit{projection formula}. The elements of $(A_1\otimes^M\cdots\otimes^M A_r)(k)$ are traditionally denoted as symbols $\{a_1,\ldots,a_r\}_{L/k}$, for a given finite extension $L/k$ and $a_i\in A_i(L)$, $i\in\{1,\ldots,r\}$. Using the fact that $N_{F/L}\circ \res_{F/L}=[F:L]$, an immediate consequence of the projection formula is that $[F:L]\{a_1,\ldots,a_r\}_{L/k}=\{\res_{F/L}(a_1),\ldots,\res_{F/L}(a_r)\}_{F/k}$ for any tower $F/L/k$ of finite extensions. In particular, symbols are preserved under isomorphisms of extensions of $k$.

		Below we give the definition of the Somekawa $K$-group $K(k;A_1,\ldots,A_r)$ attached to abelian varieties $A_1,\ldots, A_r$. For the more general definition we refer to \cite[p.~107]{Somekawa1990}. 
		\black 
		\begin{defn} (\cite[p.~107]{Somekawa1990}) The Somekawa $K$-group $K(k;A_1,\ldots,A_r)$ attached to $A_1,\ldots,A_r$ is defined to be \[K(k;A_1,\ldots,A_r):=(A_1\otimes^M\cdots\otimes^M A_r)(k)/(\mathbf{WR}),\] where $(\mathbf{WR})$ is the subgroup generated by the following types of elements. Let $C$ be a smooth projective curve over $k$. 
		For each closed point $x\in C$,  let $\iota_x\in C(k(x))$ be the canonical inclusion $\Spec k(x)\to C$, 
		and let the extension $k\to k(x)$ be determined by composing $\iota_x$ with the structure morphism $C\to\Spec k$. 
	For every $f\in k(C)^\times$	and every collection of regular maps $g_i:C\to A_i$ for $i\in\{1,\ldots,r\}$ 
	  we insist that \black 
			
			\begin{equation}\label{WR}
				\sum_{x\in C}\ord_x(f) \{g_1\circ \iota_x,\ldots, g_r\circ \iota_x\}_{k(x)/k}\in(\mathbf{WR}).
			\end{equation} 
		\end{defn}
		The generators of $K(k;A_1,\ldots,A_r)$ will again be denoted as symbols $\{a_1,\ldots,a_r\}_{L/k}$. The relation \eqref{WR} is known as \textit{Weil reciprocity}. 
	When $r=1$ we have an isomorphism $K(k;A)\simeq A(k)$. When $k$ is algebraically closed, this follows by \cite[Chapter III, Theorem 1]{Serre1988}). For a proof in the general case see \cite[Corollary 3.7]{Gazaki2015}. \black 	
		
		\begin{notn}\label{symbk}
			We will denote by  $\Symb_k(A_1,\ldots,A_r)$ the subgroup of $K(k;A_1,\ldots,A_r)$ generated by symbols $\{a_{1},\ldots,a_{r}\}_{k/k}$ defined over $k$. 
		\end{notn}

		\subsection{The norm and restriction maps on $K$-groups}\label{normres}		 
		Let $L/k$ be a finite extension. There is a restriction map on $K$-groups \[\res_{L/k}:K(k;A_1,\ldots,A_r)\rightarrow K(L;(A_1)_{L},\ldots, (A_r)_{L})\] 
		defined as follows. Let $F/k$ be a finite extension. Since we assumed that the base field $k$ is perfect, the primitive element theorem implies that the extension $F/k$ is simple. That is, $F\simeq k[X]/(f(X))$ for some irreducible polynomial $f(X)\in k[X]$ (where restricting the isomorphism to the constant field gives the embedding $k\hookrightarrow F$). This gives an isomorphism 
		\[\displaystyle F\otimes_k L\simeq L[X]/(f(X))\simeq\prod_{j=1}^n L[X]/(f_j(X))=\prod_{j=1}^n L_j,\]
		where $f(X)=\prod_j f_j(X)$ is the factorization of $f(X)$ into irreducible polynomials in $L[X]$.  For each $j$, the given extensions $L/k$ and $F/k$ induce canonical extensions $L_j/L$ (by inclusion of scalars $L\to L[X]/(f_j(X))$) and $L_j/F$ (by $k[X]/(f(X))\to L[X]/(f(X))\to L[X]/(f_j(X))$, sending $x\mapsto x$ and constants along the inclusion $k\hookrightarrow L$). 
		Given $a_i\in A_i(F)$ for $i\in\{1,\ldots,r\}$,  we define 
		\begin{align}\label{ressymbol}
			\res_{L/k}(\{a_1,\ldots,a_r\}_{F/k})=\sum_{j=1}^n\{\res_{L_j/F}(a_1),\ldots,\res_{L_j/F}(a_r)\}_{L_j/L}.
		\end{align}
		We refer to \cite[(1.3)]{Somekawa1990} for a definition of the restriction map in the imperfect case. 
		
		Moreover, there is a norm map 
		\[N_{L/k}:K(L;(A_1)_{L},\ldots, (A_r)_{L})\rightarrow K(k;A_1,\ldots,A_r)\] sending a symbol $\{b_1,\ldots,b_r\}_{F/L}$ to $\{b_1,\ldots,b_r\}_{F/k}$,  with the extension $F/k$ determined by composing the  embeddings $k\hookrightarrow L\hookrightarrow F$.  As usual, we have an equality $N_{L/k}\circ \res_{L/k}=[L:k]$.

		\subsection{$K$-group of a product of curves}\label{sec:F2iso}
		From now on we assume that $X=C_1\times\cdots\times C_d$ is a product of smooth projective curves over $k$, with fixed $k$-rational points $O_i\in C_i(k)$ for each $i\in\{1,\ldots,d\}$. To each nonempty subset $\{i_1,\ldots,i_\nu\}\subseteq\{1,\ldots,d\}$, we have a Somekawa $K$-group $K(k;J_{i_1},\ldots,J_{i_\nu})$ associated to the Jacobians $J_{i_1},\ldots,J_{i_\nu}$.
		
		In this setting, Raskind and Spiess (\cite[Corollary 2.4.1]{Raskind/Spiess2000}) established an isomorphism  
		
		\begin{equation}\label{RS}\widehat{\varepsilon}:\Z\oplus\bigoplus_{\nu=1}^d\bigoplus_{1\leq i_1<\cdots<i_\nu\leq d} K(k;J_{i_1},\ldots,J_{i_\nu})\xrightarrow{\simeq} \CH_0(X).
		\end{equation}
		
		The map $\widehat{\varepsilon}$ is quite explicit. Given a finite extension $k\hookrightarrow L$ and a symbol $\{a_{i_1},\ldots,a_{i_\nu}\}_{L/k}\in K(k;J_{i_1},\ldots,J_{i_\nu})$, define $z_i\in \CH_0((C_i)_{L})$ by $z_i:=a_i$ if $i\in\{i_1,\ldots,i_\nu\}$, and $z_i:=[O_i]_L$ otherwise. Then 
		
		\[\widehat{\varepsilon}(\{a_{i_1},\ldots,a_{i_\nu}\}_{L/k}):=({(\pi_{L/k})_\star}\circ\varepsilon_{X_L})(z_1\otimes\cdots\otimes z_d)={ (\pi_{L/k})_\star}(\pi_1^\star(z_1)\cdot\pi_2^\star(z_2)\cdot\cdots\cdot \pi_d^\star(z_d)),\] \black 
		where $\CH_0(X_L)\xrightarrow{(\pi_{L/k})_\star}\CH_0(X)$ is the pushforward map induced by base change along $L/k$. 

One can then define a finite filtration on $\CH_0(X)$ by setting $\Fil^0(X)=\CH_0(X)$, and 
\[\Fil^n(X):=\widehat{\varepsilon}\left(\bigoplus_{\nu=n}^d\bigoplus_{1\leq i_1<\cdots<i_\nu\leq d} K(k;J_{i_1},\ldots,J_{i_\nu})\right)\]	for $1\leq n\leq d$ (see \cite[Example 2.2]{Yamazaki2005}). 	This gives identifications $\Fil^1(X)=F^1(X)=\ker(\deg)$, $\Fil^2(X)=F^2(X)=\ker(\alb_X)$, and we obtain an isomorphism 
\begin{eqnarray}\label{F2iso}\bigoplus_{\nu=2}^d\bigoplus_{1\leq i_1<\cdots<i_\nu\leq d} K(k;J_{i_1},\ldots,J_{i_\nu}) \simeq F^2(X).\end{eqnarray} 
\black

		
		Recall that $\Symb_k(J_{i_1},\ldots,J_{i_\nu})$ denotes the subgroup of $K(k;J_{i_1},\ldots,J_{i_\nu})$ generated by symbols of the form $\{a_{i_1},\ldots,a_{i_\nu}\}_{k/k}$ (cf. \autoref{symbk}). The following lemma follows easily by the explicit description of the map $\widehat{\varepsilon}$. 
		\begin{lem}\label{RSlem}  The isomorphism $\widehat{\varepsilon}$ induces identifications \[\Z\oplus\bigoplus_{\nu=1}^d\bigoplus_{1\leq i_1<\cdots<i_\nu\leq d} \Symb_k(J_{i_1},\ldots,J_{i_\nu}) \simeq \CH_0(X)\new\] and 
				\begin{align}\label{F2compiso}
			F^2(X)\new\simeq\bigoplus_{\nu=2}^d\bigoplus_{1\leq i_1<\cdots<i_\nu\leq d} \Symb_k(J_{i_1},\ldots,J_{i_\nu}).
		\end{align}
		\end{lem}
		\begin{proof} We only prove the first claim. The second follows in a similar manner. \black 
		It is clear by definition that an element in $\Z$ or in $\Symb_k(J_{i_1},\ldots,J_{i_\nu})$ maps into $\CH_0(X)\new$. Conversely, write an arbitrary element $z_1\otimes\cdots\otimes z_d\in\CH_0(C_1)\otimes\cdots \otimes\CH_0(C_d)$ as
			\[\left((z_1-\deg(z_1)[O_1])+\deg(z_1)[O_1]\right)\otimes\cdots\otimes \left((z_d-\deg(z_d)[O_d])+\deg(z_d)[O_d]\right)\]
			and expand. Each term in the expansion will have in the $i$-th component either a multiple of $[O_i]$ or an element of $J_i(k)$, and so its image under $\varepsilon_X$ will be either in $\widehat{\varepsilon}(\Z)$ or in $\widehat{\varepsilon}(\Symb_k(J_{i_1},\ldots,J_{i_\nu}))$ for some choice of $i_1,\ldots,i_\nu$. 
	
		\end{proof}

\black 		

		\begin{rem}\label{compconjtofullconj} All the above discussion shows that if the base change $X_L$ satisfies the weaker \autoref{bbcompconj} for all finite extensions $L/k$, then $F^2(X)$ is generated by torsion elements, so $X$ satisfies \autoref{bbconj} \black 
		\end{rem}	 	

		\vspace{3pt}
		

		\section{Product Formula and Applications}\label{productformulasection}	
		\begin{prop}\label{mainprop} Let $A_1,\ldots, A_r$ be abelian varieties over a perfect field $k$ with $r\geq 3$. Let $L/k$ be a finite extension. There is a well-defined homomorphism  
\begin{eqnarray*}\Psi_k:&& K(k;A_1,A_2)\otimes A_3(k)\otimes\cdots\otimes A_r(k)\to K(k;A_1,\ldots, A_r) \\
&&\{a_1,a_2\}_{F/k}\otimes a_3\otimes\cdots\otimes a_r\mapsto\{a_1,a_2,\res_{F/k}(a_3)\ldots,\res_{F/k}(a_r)\}_{F/k},
\end{eqnarray*}	which extends for every finite extension $L/k$ to a homomorphism 
\[\Phi_L:K(k;A_1,A_2)\otimes A_3(L)\otimes\cdots\otimes A_r(L)\to K(k;A_1,\ldots, A_r)\] given by $\Phi_L=N_{L/k}\circ\Psi_L\circ(\res_{L/k}\otimes 1_{A_3}\otimes\cdots\otimes 1_{A_r})$, where $N_{L/k}$ and $\res_{L/k}$ are the norm and restriction maps on Somekawa $K$-groups (see \autoref{normres}). \black 
		 
		\end{prop} 
		\begin{proof} It is enough to show that the homomorphism $\Psi_k$ is well-defined. \black 
				We first verify the projection formula \eqref{PF}. Let $F/K/k$ be a tower of finite extensions, and suppose $a_1\in A_1(F), a_2\in A_2(K)$. Moreover, let $a_i\in A_i(k)$ for $3\leq i\leq r$. We compute
			\begin{eqnarray*}
				&&\Psi_k(\{a_1,\res_{F/K}(a_2)\}_{F/k}\otimes a_3\otimes\cdots\otimes a_r)=
				\{a_1,\res_{F/K}(a_2),\res_{F/k}(a_3),\ldots,\res_{F/k}(a_r)\}_{F/k}=\\
				&&\{a_1,\res_{F/K}(a_2),\res_{F/K}(\res_{K/k}(a_3)),\ldots,\res_{F/K}(\res_{K/k}(a_r))\}_{F/k}\stackrel{(\mathbf{PF})}{=}\\
				&&\{N_{F/K}(a_1),a_2, \res_{K/k}(a_3),\ldots,\res_{K/k}(a_r)\}_{K/k}=\Psi_k(\{N_{F/K}(a_1),a_2\}_{K/k}\otimes a_3\otimes\cdots\otimes a_r).
			\end{eqnarray*}
			The symmetric relation for $a_1\in A_1(K), a_2\in A_2(F)$ is analogous. 
			
			We next verify Weil reciprocity \eqref{WR}. Let $C$ be a smooth projective curve over $k$ and suppose we have regular maps $g_i:C\to A_i$ for $i=1,2$.  Let $a_i\in A_i(k)$ for $i\in\{3,\ldots,r\}$. Let $C\xrightarrow{\pi}\Spec k$ be the structure morphism. For each $i\in\{3,\ldots,r\}$ define $g_i:C\to A_i$ to be the constant map $g_i:=a_i\circ \pi$. Let $f\in k(C)^\times$. \black  We have
			
\begin{eqnarray*} 
				&&\Phi_k\left(\sum_{x\in C}\ord_x(f)\{g_1\circ\iota_x,g_2\circ\iota_x\}_{k(x)/k} \otimes a_3\otimes\cdots \otimes a_r\right)=\\
				&& \sum_{x\in C}\ord_x(f)\{g_1\circ\iota_x,g_2\circ\iota_x, \res_{k(x)/k}(a_3),\ldots,\res_{k(x)/k}(a_r)\}_{k(x)/k}=\\
				&&\sum_{x\in C}\ord_x(f)\{g_1\circ\iota_x,g_2\circ\iota_x, g_3\circ\iota_x,\ldots,g_r\circ\iota_x\}_{k(x)/k}\stackrel{(\mathbf{WR})}{=}0.
			\end{eqnarray*}

		\end{proof}
		
		\begin{rem} \autoref{mainprop} can be thought of as the analog of the product formula for Milnor $K$-groups. Kahn and Yamazaki (cf. \cite{Kahn/Yamazaki2013}) in fact defined a tensor product structure in the category of homotopy invariant Nisnevich sheaves with transfers, which when evaluated at $\Spec(k)$ coincides with a Somekawa $K$-group. Such a property therefore is not surprising.   
		\end{rem}
		
		\autoref{mainintro} follows now as an easy corollary to \autoref{mainprop}. We restate this in greater generality here. 
		\begin{cor}\label{main} Let $X=C_1\times\cdots\times C_d$ be a product of smooth projective curves over a perfect field $k$  with $C_i(k)\neq \emptyset$ for $i\in\{1,\ldots,d\}$.  Suppose that for each $1\leq i<j\leq d$ the group $F^2(C_i\times C_j)$ is torsion. Then $F^2(X)$ is torsion. In the special case when $k$ is an algebraic number field, we conclude that if \autoref{bbconj} is true for any subproduct of two curves, then it is true for $X$. 
		\end{cor}
		
		\begin{proof} 
			The assumption that $F^2(C_i\times C_j)$ is torsion for each $1\leq i< j\leq d$ means precisely that the group 
			$\displaystyle\bigoplus_{1\leq i< j\leq d}K(k;J_i,J_j)$ is torsion. If $d=2$ then we are done, so let us assume $d\geq 3$. Using the identification \eqref{F2iso}, it is enough to show that for every $3\leq\nu\leq d$ the group $K(k;J_{i_1},\ldots,J_{i_\nu})$ is torsion. Thus, the claim comes down to showing that if $A_1,\ldots, A_r$ are abelian varieties over $k$ such that the Somekawa $K$-group $K(k;A_1,A_2)$ is torsion, then so is the group $K(k;A_1,\ldots,A_r)$.  
			
			Let $\{a_1,\ldots, a_r\}_{L/k}\in K(k;A_1,\ldots,A_r)$, where $L/k$ is a finite extension and $a_i\in A_i(L)$ for $i\in\{1,\ldots,r\}$.
			If $\widehat{L}/k$ is the  Galois \black  closure of $L/k$, then \black 
			\begin{align*}
				[\widehat{L}:L]\{a_1,\ldots, a_r\}_{L/k}&\stackrel{(\mathbf{PF})}{=}\{\res_{\widehat{L}/L}(a_{i_1}),\res_{\widehat{L}/L}(a_{1}),\ldots, \res_{\widehat{L}/L}(a_{r})\}_{\widehat{L}/k}.
			\end{align*}
			Thus, to prove $\{a_1,a_2,\ldots, a_r\}_{L/k}$ is torsion, it is enough to consider the case that $L/k$ is a Galois extension with Galois group $G$. \black 
						We consider the homomorphism 
			\[\Phi_L:K(k;A_1,A_2)\otimes A_3(L)\otimes\cdots\otimes A_r(L)\to K(k;A_1,\ldots, A_r)\] of \autoref{mainprop}. To show that the symbol $\{a_1,a_2,\ldots, a_r\}_{L/k}$ is torsion, it suffices to show that an integer multiple of it lies in the image of $\Phi_L$. We have, 
\begin{eqnarray*}
&&\Phi_L(\{a_1,a_2\}_{L/k}\otimes a_3\otimes\cdots\otimes a_r)=\\
&&N_{L/k}\circ\Psi_L\circ(\res_{L/k}\otimes 1_{A_3}\otimes\cdots\otimes 1_{A_r})(\{a_1,a_2\}_{L/k}\otimes a_3\otimes\cdots\otimes a_r). 
\end{eqnarray*} 
Since $L/k$ is finite Galois, we have an isomorphism of $L$-algebras,
			\[L\otimes_k L\simeq\prod_{\sigma\in G}L.\] In particular, in the definition of the restriction map on Somekawa $K$-groups (see \autoref{normres}) we have $L_j\simeq L$ for all $j=1,\ldots,[L:k]$, and hence  
\begin{eqnarray*}
&&\res_{L/k}(\{a_1,a_2\}_{L/k})=\sum_{j=1}^{[L:k]}\{a_1,a_2\}_{L/L}=[L:k]\{a_1,a_2\}_{L/L}. 
\end{eqnarray*} We conclude that 
\begin{eqnarray*}
\Phi_L(\{a_1,a_2\}_{L/k}\otimes a_3\otimes\cdots\otimes a_r)=&&
N_{L/k}([L:k] \{a_1,\ldots, a_r\}_{L/L})\\=
&&[L:k]\{a_1,\ldots, a_r\}_{L/k}. 
\end{eqnarray*} 

			\black

		\end{proof}
		
		\begin{rem} The assumption $C_i(k)\neq\emptyset$ can be removed, if we instead assume that for some finite extension $L/k$, we have $C_i(L)\neq\emptyset$ for all $i$ and \autoref{bbconj} holds for every product $(C_i)_{L}\times (C_j)_{L}$ with $1\leq i<j\leq d$.   This follows easily by the formula  $N_{L/k}\circ\res_{L/k}=[L:k]$.   
		\end{rem}
		
%
%
 The previous corollary shows that \autoref{bbconj} holds for arbitrary products of curves if it holds for pairwise products; the following corollary gives the analogous statement for the weaker \autoref{bbcompconj}.
	 	
		\begin{cor}\label{corollaryS}
			Let $X=C_1\times\cdots\times C_d$ be a product of smooth projective curves over a perfect field $k$  with $C_i(k)\neq \emptyset$ for $i\in\{1,\ldots,d\}$.  Suppose that for each $1\leq i<j\leq d$ the group $F^2(C_i\times C_j)\new$ is finite. Then $F^2(X)\new$ is finite. 
		\end{cor} 
		\black
		\begin{proof}
			Since $C_i(k)\neq \emptyset$ for all $i$ the Raskind-Spiess isomorphism applies, and so by \eqref{F2compiso}, the assumption that $F^2(C_{i}\times C_{j})\new$ is finite for all $1\leq i<j\leq d$ amounts to saying that $\Symb_k(J_{i}, J_{j})$ is finite. The corollary then follows by \autoref{mainprop} after noticing the surjectivity of the homomorphism 
			\begin{align*}
				\Psi_k:\Symb_k(J_{i_1}, J_{i_2})\otimes J_{i_3}(k)\otimes\cdots\otimes  J_{i_\nu}(k)\rightarrow\Symb_k(J_{i_1},\ldots,J_{i_\nu}).
			\end{align*}
		\end{proof}
		 In the following section, we will give many applications of \autoref{corollaryS}.  
		
		\vspace{2pt}
		\section{Finiteness of $F^2(C_1\times C_2)\new$}\label{finiteoverk}

		The goal of this section is to determine pairs of curves $C_1,C_2$ such that $F^2(C_1\times C_2)\new$ is finite. Note that if $C$ is a curve then $F^2(C\times C)\new$ is not a priori finite, so the case $C_1=C_2$ is often nontrivial.

		\begin{conv} Let  $C_1,C_2$ be curves over $k$ with Jacobians $J_1$ and $J_2$, and let $X=C_1\times C_2$. We will suppose there are fixed rational points $O_1\in C_1(k)$ and $O_2\in C_2(k)$; if $C_i$ is an elliptic curve, we will always take $O_i$ to be the identity of $C_i(k)$. 
		\end{conv}
		
		There are two main principles we use in order to examine  whether $F^2(X)\new$ is finite. \autoref{finiteindex} allows us to reduce the problem to checking a finite list of relations, and \autoref{torsionfromhyper} gives us a way to produce these relations. 
		

		\begin{lem}\label{finiteindex}
			Let $B$ be a list of $(\rk J_1(k))(\rk J_2(k))$ linearly independent elements of $J_1(k)\otimes J_2(k)$. If $\varepsilon_X(t)$ is torsion for all $t\in B$, then $F^2(X)\new$ is finite.
		\end{lem}
		\begin{proof}
			By \eqref{F2compiso}, $F^2(X)\new$ is the image of the finitely generated abelian group $J_1(k)\otimes J_2(k)$ under $\varepsilon_X$. Since $B$ generates a full-rank subgroup of $J_1(k)\otimes J_2(k)$, every element of $J_1(k)\otimes J_2(k)$ maps to torsion in $F^2(X)\new$.
		\end{proof}

		We list two consequences of \autoref{finiteindex}. The first is immediate.
		
		\begin{cor}\label{rank0}
			If $J_1(k)$ or $J_2(k)$ has rank $0$, then $F^2(X)\new$ is finite.
		\end{cor}

		\begin{cor}\label{Sisogclasses}
			Let $C_1, C_2, C_2'$ be curves over $k$. If $F^2(C_1\times C_2)\new$ is finite and the Jacobians of $C_2$ and $C_2'$ are isogenous, then $F^2(C_1\times C_2')\new$ is finite.
		\end{cor}
		
		\begin{proof}
			This follows by covariant functoriality of the Somekawa $K$-group (cf.~\cite[p.~4]{Somekawa1990}). Namely, let $J_2$ and $J_2'$ be the Jacobians of $C_2$ and $C_2'$ respectively; if $\phi:J_2\to J_2'$ is an isogeny, then it induces a group homomorphism $K(k;J_1,J_2)\to K(k;J_1,J_2')$ sending a symbol $\{a,b\}_{L/k}$ to $\{a,\phi(b)\}_{L/k}$.
			
			Finiteness of $F^2(C_1\times C_2)\new$ implies that  $\{P_1,P_2\}_{k/k}\in \Symb_k(J_1,J_2)$ is torsion for all $P_1\in J_1(k)$ and $P_2\in J_2(k)$, and so 
			\[\varepsilon_X(P_1\otimes \phi(P_2))=\{P_1,\phi(P_2)\}_{k/k}\in \Symb_k(J_1,J_2')\]
			is torsion. Since the map $J_2(k)\to J_2'(k)$ induced by $\phi$ has finite cokernel, elements of the form $P_1\otimes\phi(P_2)$ generate a full-rank subgroup of $J_1(k)\otimes J_2(k)$.
			
		\end{proof}
	\black
		

		
		\vspace{1pt}
		
		Given an elliptic or hyperelliptic curve $H$, let $P\mapsto \overline{P}$ denote negation or the hyperelliptic involution, respectively. 
		\begin{prop} \label{torsionfromhyper}
			Let $H$ be an elliptic or hyperelliptic curve, let $P\in H(\overline{k})$, and let $\phi_i:H\to C_i$ be a morphism for each $i=1,2$. Suppose that $\phi_i(P)\in C_i(k)$ for each $i=1,2$, and also suppose there exists $W\in H(\overline{k})$ with $W=\overline{W}$ and such that $[\phi_i(W)]-[O_i]\in J_i(\overline{k})$ is torsion for $i=1,2$. Then the symbol
			\[\{[\phi_1(P)]-[O_1],[\phi_{2}(P)]-[O_2]\}_{k/k}\in \Symb_k(J_1,J_2)\]
			is torsion.
		\end{prop}
		\begin{proof}
			Let $L/k$ be a finite extension that contains  subfields isomorphic to the fields of definition of $W$ and of $P$, so that we can identify $P$, $\overline{P}$, and $W$ with elements of $H(L)$. Then $[P]+[\overline{P}]-2[W]$ is a principal divisor on $H_L$.

			Define $g_i:H_L\to (J_i)_L$  to act on $L$-rational points by  $R\mapsto [\phi_i(R)]-[\phi_i(W)]$. For $i=1,2$, set $W_i=\phi_i(W)$ and $P_i:=\phi_i(P)$. Note that the relation $[\overline{P}]-[W]=-[P]+[W]$ on $H$ induces the relation 
			\[[\phi_i(\overline{P})]-[\phi_i(W)]=-[\phi_i(P)]+[\phi_i(W)]=-[P_i]+[W_i]\]
			in $J_i(L)$ by proper pushforward of Chow groups. The principal divisor $[P]+[\overline{P}]-2[W]$ then induces a Weil reciprocity relation \eqref{WR} on the group $K(L;(J_1)_L,(J_2)_L)$:
			\begin{align*}
				0&= \{[P_1]-[W_1],\;[P_2]-[W_2]\}_{L/L} +\{-[P_1]+[W_1],\;-[P_2]+[W_2]\}_{L/L}-2\{0,\;0\}_{L/L}\\
				&=2\{[P_1]-[W_1],\;[P_2]-[W_2]\}_{L/L}.
			\end{align*}
			
			Since $[W_i]-[O_i]\in J_{1L}(L)$  is torsion for each $i=1,2$, there exists $n>0$ such that 
			\begin{align*}
				0&=n\{[P_1]-[W_1],[P_2]-[W_2]\}_{L/L}\\
				&=n\{[P_1]-[O_1]+[O_1]-[W_1],[P_2]-[O_2]+[O_2]-[W_2]\}_{L/L}\\
				&=n\{[P_1]-[O_1],[P_2]-[O_2]\}_{L/L}.
			\end{align*}
			Applying the norm map  $N_{L/k}:K(L;(J_1)_L,(J_2)_L)\to K(k;J_1,J_2)$  (see \autoref{normres}), we can conclude that
			\[0=2n\{[P_1]_L-[O_1]_L,[P_2]_L-[O_2]_L\}_{L/k}.\]
			Since $P_i\in C_i(k)$ by assumption, the divisors $[P_i]_L-[O_i]_L\in J_i(L)$ are in the image of $\res_{L/k}:J_i(k)\to J_i(L)$. Hence, the projection formula \eqref{PF} gives
			\begin{align*}0&=n\{\res_{L/k}([P_1]-[O_1]),\res_{L/k}([P_2]-[O_2])\}_{L/k}\\
				&=n\{[P_1]-[O_1],N_{L/k}(\res_{L/k}([P_2]-[O_2]))\}_{k/k}\\
				&=n[L:k]\{[P_1]-[O_1],[P_2]-[O_2]\}_{k/k}. 
			\end{align*}	
			
			
		\end{proof}

		\begin{rem}
			Most of the examples we explore below use \autoref{torsionfromhyper}; that is, the key relation comes from Weil reciprocity for an elliptic or hyperelliptic curve. See \autoref{bielliptic} for examples that use Weil reciprocity for a curve that is neither elliptic nor hyperelliptic.
		\end{rem}
		

		\subsection{Isogenous pairs}
		
		We list some consequences of \autoref{torsionfromhyper} that can be obtained by setting $H$ equal to one of the factors of $X$.
		\subsubsection{Rank 1 hyperelliptic curves}
		
		\begin{prop}\label{hyperselfproduct}
			Let $C_1,C_2$ be hyperelliptic curves over $k$ with isogenous Jacobians $J_1$ and $J_2$. Suppose $\rk J_1(k)=1$, and there exist $P,W\in C_1(k)$ with $W$ a Weierstrass point and $[P]-[W]$ of infinite order. Then $F^2(C_1\times C_2)\new$ is finite.
		\end{prop}
		\begin{proof}
			For the case $C_2=C_1$, we can apply \autoref{torsionfromhyper} and \autoref{finiteindex} using $\phi_1=\phi_2$ the identity map on $H$, and $O_1=O_2=W$. The general case follows by \autoref{Sisogclasses}.
		\end{proof}
		
		In light of this result, \autoref{rank0}, and \autoref{corollaryS}, we can conclude that $F^2(C_1\times\cdots \times C_d)\new$ is finite if all $J_i$ with $\rk J_i(k)\neq 0$ come from a single isogeny class with rank $1$ over $k$.
		
		While \autoref{hyperselfproduct} does apply when $C_1,C_2$ are isogenous elliptic curves (take $W$ to be the identity of $C_1(k)$), it yields a more interesting result when the genus of $C_1$ is at least $2$. \black In the LMFDB~\cite{lmfdb} there are $868$ genus $2$ curves over $\Q$ that have Jacobian of rank $1$, conductor at most $10000$, and a rational Weierstrass point $W$. For all but $8$ of these, there exists another rational point $P$ such that $[P]-[W]$ has infinite order. 
		These provide the first nontrivial examples of products $X$ with $F^2(X)\new$ finite involving positive rank curves of higher genus.

		\subsubsection{Self-product of a rank 2 CM elliptic curve}
		
		Let $E$ be an elliptic curve over $\Q$ such that $E_{\overline{\Q}}$ has complex multiplication by the full ring of integers $\mathcal{O}_K$ of a quadratic imaginary field $K$. It follows that $K$ is of the form $K=\Q(\sqrt{-D})$ with $D\in\{1,2,3,7,11,19, 43, 67, 163\}$. We write $\mathcal{O}_K=\Z[\omega]$ and we denote by $[\omega]:E_K\to E_K$ the extra endomorphism. Suppose that $E(\Q)$ has rank $1$ and let $P$ be a point of infinite order. It follows by \cite[Lemma 4.5]{gazaki2021weak} that the point $[\omega](P)\in E(K)$ is $\Z$-linearly independent from $P$, and thus $E(K)$ has rank at least $2$. In all cases we have computed, $E(K)$ actually has rank exactly $2$. 
		\begin{prop}\label{CMcase} Suppose that $E$ satisfies the above assumptions and assume additionally that $\rk(E(\Q))=1$ and $\rk(E(K))=2$. Let $S$ be a set consisting of $E_K$ together with arbitrarily many curves over $K$ with Jacobian of rank $0$. Then $F^2(C_1\times\cdots\times C_d)\new$ is finite for all $C_1,\cdots, C_d\in S$. (In particular, $F^2(E_K\times E_K)\new$ is finite.)
		\end{prop}
		\begin{proof}  By \autoref{corollaryS} it suffices to show that $F^2(C_i\times C_j)\new$ is finite for all $1\leq i<j\leq d$. If either $J_i$ or $J_j$ has rank $0$ this follows by \autoref{rank0}. The only remaining case is when $C_i=C_j=E_K$. 
		\black 
			Since $P\in E(\Q)$ and $[\omega](P)\in E(K)\setminus E(\Q)$ are $\Z$-linearly independent, they generate a finite index subgroup of $E(K)$. By \autoref{finiteindex}, it is enough to show that the symbols \[\{P,P\}_{K/K},\;\{[\omega](P),P\}_{K/K},\;\{P,[\omega](P)\}_{K/K},\;\{[\omega](P),[\omega](P)\}_{K/K}\]
			are all torsion. This follows by applying \autoref{torsionfromhyper} to the four choices of $\phi_1,\phi_2\in\{[1],[\omega]\}$. 
		\end{proof}
		
		\subsection{Non-isogenous pairs: a new construction}\label{newconstruction}
		
		\black
		Let $E_1,E_2$ be elliptic curves over $k$ with fully $k$-rational $2$-torsion, so that there exist Weierstrass forms
		\[E_1\simeq E_{\alpha,\beta}:y^2=x(x-\alpha)(x-\beta),\qquad E_2\simeq E_{\gamma,\delta}:y^2=x(x-\gamma)(x-\delta)\]
		for some $\alpha,\beta,\gamma,\delta\in k^\times$ with $\alpha\neq\beta$ and $\gamma\neq \delta$. In~\cite[Theorem 1]{Scholten}, Scholten defines the curve
		\begin{align}\label{scholtencurve}
			H_{\alpha,\beta,\gamma,\delta}:(\alpha\delta-\beta\gamma)y^2=((\alpha-\beta)x^2-(\gamma-\delta))(\alpha x^2-\gamma)(\beta x^2-\delta)
		\end{align}
		over $k$, and provides explicit degree $2$ $k$-rational maps $f_1:H_{\alpha,\beta,\gamma,\delta}\to E_{\alpha,\beta}$ and $f_2:H_{\alpha,\beta,\gamma,\delta}\to E_{\gamma,\delta}$. These maps each send the Weierstrass points of $H_{\alpha,\beta,\gamma,\delta}$ to $2$-torsion points.
		Composing $f_1$ with an isomorphism $E_{\alpha,\beta}\to E_1$, and $f_2$ with an isomorphism $E_{\gamma,\delta}\to E_2$, we obtain maps $\phi_i:H_{\alpha,\beta,\gamma,\delta}\to E_i$, which map Weierstrass points to torsion points. This means that any $P\in H(k)$ will allow us to apply \autoref{torsionfromhyper}.

		\begin{algorithm}
			\caption{Computing elements of $E_1(k)\otimes E_2(k)$ that map to torsion in $F^2(E_1\times E_2)$.}
			\label{tensoralg}
			\SetAlgoLined
			\DontPrintSemicolon
			\SetKwInOut{Input}{Input}\SetKwInOut{Output}{Output}
			\Input{$\alpha,\beta,\gamma,\delta\in k^\times$ with $\alpha\neq\beta$ and $\gamma\neq \delta$.}
			\Output{A list of pairs $(P_1,P_2)\in E_1(k)\times E_2(k)$, for $E_1:y^2=x(x-\alpha)(x-\beta)$ and $E_2:y^2=x(x-\gamma)(x-\delta)$.}
			
			Find a basis $R_1,\ldots,R_r$ for a full-rank subgroup of $E_1(k)$, and a basis $S_1,\ldots,S_s$ for a full-rank subgroup of $E_2(k)$.\label{getbases}\;
			Initialize empty lists $B$ (of points in $E_1(k)\times E_2(k)$) and $T$ (of vectors in $\Q^{rs}$).\;
			Define $H:=H_{\alpha,\beta,\gamma,\delta}$ as in \eqref{scholtencurve} and the corresponding maps $\phi_i:H\to E_i$.\;
			Search for rational points $P\in H(k)$.\;
			\For{each point $P\in H(k)$ found} {
				Find $a_1,\ldots,a_r,d_1\in\Z$, $d_1\neq 0$,  such that $a_1R_1+\cdots +a_rR_r=d_1\phi_1(P)$ in $E_1(k)$.\label{phi1Pcoeffs}\;
				Find $b_1,\ldots,b_s,d_2\in\Z$, $d_2\neq 0$,  such that $b_1S_1+\cdots +b_sS_s=d_2\phi_2(P)$ in $E_2(k)$.\label{phi2Pcoeffs}\;
				Define the vector $v(P)\in \Q^{rs}$ to be \label{tensormap}
				\[\hspace{-50pt}
				 \frac{1}{d_1d_2}  (a_1,\ldots,a_r)\otimes (b_1,\ldots,b_s)=\frac{1}{d_1d_2} \black (a_1b_1,a_1b_2,\ldots,a_1b_s,a_2b_1,\ldots,a_2b_s,\ldots, a_rb_s).\]
				
				\If {$T\cup \{v(P)\}$ is linearly independent} {
					Append $\black v(P)$ to $T$.\;
					Append $(\phi_1(P),\phi_2(P))$ to $B$.\label{Lappend}
				}
			}
			Return $B$.\;
		\end{algorithm}
		
		\begin{prop}\label{algproof}
			Let $\alpha,\beta,\gamma,\delta\in k^\times$ with $\alpha\neq\beta$ and $\gamma\neq \delta$, and set $E_1:y^2=x(x-\alpha)(x-\beta)$ and $E_2:y^2=x(x-\gamma)(x-\delta)$. Let $B$ be  the list of points in $E_1(k)\times E_2(k)$  output by \autoref{tensoralg} below. Then $(P_1\otimes P_2:(P_1,P_2)\in B)$ is a linearly independent list of elements of $E_1(k)\otimes E_2(k)$, all of which map to torsion in $F^2(E_1\times E_2)$ under $\varepsilon_X$. 
		\end{prop}
		\begin{proof}
			An element $(P_1,P_2)\in B$ is of the form $(\phi_1(P),\phi_2(P))$ for some $P\in H(k)$. Recall that the maps $\phi_i$ send Weierstrass points of $H$ to $2$-torsion points of $E_1$ and $E_2$. Hence, by \autoref{torsionfromhyper}, $\varepsilon_X(P_1\otimes P_2)$ is torsion. Now $v(P)$, defined in \autoref{tensormap}, is the image of $(\phi_1(P),\phi_2(P))$ under the bilinear map 
			\[E_1(k)\times E_2(k)\to (E_1(k)\otimes E_2(k))\otimes\Q\simeq \Q^{rs}\]
			induced by sending $(R_i,S_j)\mapsto R_i\otimes S_j$ for each $1\leq i\leq r$ and $1\leq j\leq s$. We can see from \autoref{Lappend} and the preceding lines that $B$ is constructed so that its image $T$ in $(E_1(k)\otimes E_2(k))\otimes\Q$ (and therefore its image in $E_1(k)\otimes E_2(k)$) is a linearly independent list. \black
		\end{proof}
		
		Applying \autoref{finiteindex} gives the following immediate consequence.
		
		\begin{cor}\label{algfinite}
			Assume the setup of \autoref{algproof}. If the length of $B$ is equal to $(\rk E_1(k))(\rk E_2(k))$, then $F^2(E_1\times E_2)\new$ is finite.
		\end{cor}
		
		\subsubsection{Computations}
		
		The authors used Magma to implement a version of \autoref{tensoralg}~\cite{Lovecode},  with a few modifications as  described in \autoref{codechanges}. 
		
		By running this algorithm and applying \autoref{algfinite}, the authors were able to find a substantial collection of pairs of curves $E_1\times E_2$ such that $F^2(E_1\times E_2)\new$ is finite, including the first known examples of pairs of non-isomorphic curves with rank greater than $1$. Specifically, they used the LMFDB~\cite{lmfdb} to obtain a list of the first $m_1$ elliptic curves over $\Q$ of rank $r_1$ and torsion subgroup $C_2\times C_2$ (ordered by conductor), and the first $m_2$ elliptic curves over $\Q$ of rank $r_2$ and torsion subgroup $C_2\times C_2$, for various choices of $m_1,r_1,m_2,r_2$. The torsion subgroup $C_2\times C_2$ was chosen as it is the simplest torsion structure for which the curves have fully rational $2$-torsion, though the algorithm would work equally well on any torsion structure of the form $C_2\times C_{2n}$. \black For each $E_1$ from the first list and $E_2$ from the second, they ran \autoref{tensoralg}; for each pair with an output of length $r_1r_2$, the pair $(E_1,E_2)$ was added to a list of pairs for which $F^2(E_1\times E_2)\new$ is known to be finite.
		
		The findings are summarized in \autoref{data}, with $m_1,r_1,m_2,r_2$ appearing as ``\# of $E_1$,'' ``$\rk E_1$,'' ``\# of $E_2$,'' and ``$\rk E_2$,'' respectively. For rows with $\rk E_1(\Q)=\rk E_2(\Q)$, only one pair out of $(E_1,E_2)$ or $(E_2,E_1)$ was tested, and diagonal pairs $(E_1,E_1)$ were not included. For rows containing ``$=E_1$'' in the ``\# of $E_2$'' column, only pairs $(E_1,E_1)$ were tested.
		
		\begingroup
		\renewcommand*{\arraystretch}{1.1}
		\begin{table}
			\begin{tabular}{| c | c | c | c | c | c |} 
				\hline
				\# of $E_1$ & $\rk E_1(\Q)$ & \# of $E_2$ & $\rk E_2(\Q)$ & total \# of pairs & \# $F^2(X)\new$ finite \\
				\hline
				100 & 1 & 100 & 1& 4950 & 2602 \\
				100 & 2 & 100 & 2& 4950 & 995 \\
				100 & 2 & $=E_1$ & & 100 & 70 \\
				20 & 3 & 20 & 3& 190 & 17 \\
				20 & 3 & $=E_1$ &   & 20& 8\\
				100 & 1 & 100 & 2& 10000 & 3311 \\
				500 & 1 & 20 & 3 & 10000 & 955\\
				500 & 2 & 20 & 3 & 10000 & 615\\
				\hline
			\end{tabular}
			\caption{Counting pairs of elliptic curves over $\Q$ with $F^2(E_1\times E_2)\new$ finite}\label{data}
		\end{table}
		\endgroup
		
		\subsubsection{Improvements on \autoref{tensoralg}}\label{codechanges}
		
		The authors' code makes some changes to \autoref{tensoralg}, including the following.
		\begin{itemize}	
			\item The curve $H_{\alpha,\beta,\gamma,\delta}$ depends on the choice of $\alpha,\beta,\gamma,\delta$, and is not invariant under isomorphism. In fact, while the curves 
			\begin{align}
				E_{\alpha,\beta}\simeq E_{\beta,\alpha}\simeq E_{-\alpha,\beta-\alpha}\simeq E_{\beta-\alpha,-\alpha}\simeq E_{-\beta,\alpha-\beta}\simeq E_{\alpha-\beta,-\beta}
			\end{align}
			are isomorphic, these typically produce six\footnote{There may be fewer, for example because of accidental isomorphisms, or because $\alpha\delta-\beta\gamma=0$ so that \eqref{scholtencurve} defines a degenerate curve.
				There are of course many other tuples $(\alpha',\beta',\gamma',\delta')$ with isomorphisms $E_{\alpha',\beta'}\simeq E_{\alpha,\beta}$ and $E_{\gamma',\delta'}\simeq E_{\gamma,\delta}$, but one can check in each case that the resulting curve $H_{\alpha',\beta',\gamma',\delta'}$ is isomorphic to one of the six curves from \eqref{sixhypers}, and that the corresponding maps are compatible in an appropriate sense.} non-isomorphic hyperelliptic curves
			\begin{align}\label{sixhypers}
				H_{\alpha,\beta,\gamma,\delta},\; H_{\beta,\alpha,\gamma,\delta},\; H_{-\alpha,\beta-\alpha,\gamma,\delta},\; H_{\beta-\alpha,-\alpha,\gamma,\delta},\; H_{-\beta,\alpha-\beta,\gamma,\delta},\; H_{\alpha-\beta,-\beta,\gamma,\delta}.
			\end{align}
			Each of these six curves $H$ comes with maps $\phi_1:H\to E_1$ and $\phi_2:H\to E_2$. Therefore, we can extend \autoref{tensoralg} to search for points $P$ on all six of these curves. 
			
			\item As written, \autoref{getbases} requires a full set of Mordell-Weil generators for $E_1(k)$ and $E_2(k)$. This is not a problem if generators are given (for instance if the curves are obtained from the LMFDB), but if they are not already known then computing these sets can potentially be a significant bottleneck for the runtime. \black Fortunately, a full set of generators is not necessary. Instead, one can initialize empty lists $B_i$ of elements of $E_i(k)$ for $i=1,2$. For each $P\in H(k)$, one checks whether \autoref{phi1Pcoeffs} is possible using the current points in $B_1$. If not, first append $\phi_1(P)$ to $B_1$, and then proceed. One can do the same for $B_2$ and \autoref{phi2Pcoeffs}.
			
			With this modification, the values of $r=\#B_1$ and $s=\#B_2$ are regularly updated, and $T$ must be recomputed from $B$ whenever this occurs. It is therefore helpful to define new lists to store the coefficients of $\phi_1(P)$ (resp.\ $\phi_2(P)$) generated in \autoref{phi1Pcoeffs} (resp.\ \autoref{phi2Pcoeffs}) for each $P$ considered, in order to more quickly recompute $v(P)\in \Q^{rs}$ for the new values of $r$ and $s$. Note that the ranks of $E_1(k)$ and $E_2(k)$ must be known in advance in order to know when we reach $r=\rk E_1(k)$ and $s=\rk E_2(k)$.
			
			\item Rather than returning a list of points $(\phi_1(P),\phi_2(P))\in E_1(k)\times E_2(k)$, the code returns a list of tuples $(P,\phi_1,\phi_2)$, so that Weil reciprocity \eqref{WR} can be verified.
		\end{itemize}

		\subsection{Non-isogenous pairs: previous constructions}\label{previousconstructions}
		We are aware of two previous constructions of pairs $E_1,E_2$ of non-isogenous rank $1$ curves over $k$ with $F^2(E_1\times E_2)\new$ finite. 
		Both constructions use the fact that it suffices to find non-torsion points $P_1\in E_1(k)$ and $P_2\in E_2(k)$ with $\{P_1,P_2\}_{k/k}$ torsion (\autoref{finiteindex}). \black

		\subsubsection{Modular parametrizations}\label{prasannasrinivas}
		
		Prasanna and Srinivas considered two pairs of non-isogenous  rank $1$  curves over $\Q$: one pair of curves of conductor $91$,
		\begin{align*}
			E_1^{(91)}&:y^2+y=x^3+x,& E_2^{(91)}&:y^2+y=x^3+x^2-7x+5,\\
			\intertext{and one pair of curves of conductor $37$,}
			E_1^{(37)}&:y^2+y=x^3-x,& E_2^{(37)}&:y^2+y=x^3+x^2-23x-50.
		\end{align*}
		\black
		They prove the following proposition. 
		\begin{prop}[\cite{Prasanna/Srinivas2022}, in preparation\black ]
			
			Let $N=37$ or $91$. Let $X=E_1^{(N)}\times E_2^{(N)}$ as defined above. Then $F^2(X)\new$ is finite. 
			
		\end{prop} 
		
		Note that for $N=37$ or $91$, every elliptic curve of conductor $N$ is isogenous to either $E_1^{(N)}$ or $E_2^{(N)}$ (see for example~\cite{lmfdb}). So together with \autoref{Sisogclasses} we can conclude that $F^2(E_1\times E_2)\new$ is finite for any two curves $E_1,E_2$ of conductor $N$.

		We make a few remarks on the approach. For each of $N=37$ and $N=91$, they show that the two modular parametrizations $X_0(N)\to E_i^{(N)}$ factor through a common hyperelliptic curve $H$. For $N=37$, $X_0(37)$ is itself a genus $2$ hyperelliptic curve. On the other hand, $X_0(91)$ is not hyperelliptic, but it has an Atkin-Lehner involution $w_{91}$, and the modular parametrizations of $E_1^{(91)}$ and $E_2^{(91)}$ both factor through the genus $2$ hyperelliptic curve $X_0(91)/\langle w_{91}\rangle$. Since the two elliptic curves have rank $1$, the theory of Heegner points can be used to compute a point $P\in H(\overline{\Q})$ such that the image of $P$ under each $f_i:H\to E_i^{(N)}$ is a rational point of infinite order. 
		
		As in \autoref{torsionfromhyper}, one can obtain the desired relation in $F^2(E_1^{(N)}\times E_2^{(N)})$ by applying Weil reciprocity to a principal divisor on $H$. Rather than using a Weierstrass point as in \autoref{torsionfromhyper}, however, Prasanna and Srinivas find another point $Q\in H(\overline{\Q})$ such that $f_i(Q)\in E_i^{(N)}(\Q)$, and apply Weil reciprocity to the divisor $[P]+[\overline{P}]-[Q]-[\overline{Q}]$.
		\black
		
		\subsubsection{Rational curves in a Kummer surface}
		
		Given a number field $k$, define a $2$-parameter family of curves over $k$ by
		\begin{equation}\label{Est}
			E_{s,t}:y^2=x^3-3t^2x+2t^3+(1-s-3t)^2s.
		\end{equation}
		
		Note that the curve $E_{s,t}$ has a $k$-point $P_{s,t}:=(1-s-2t,1-s-3t)$. For each $s\in k$, let $D_s$ denote the locus of points $t\in k$ at which the discriminant of $E_{s,t}$ vanishes, or the point $P_{s,t}$ is torsion in $E_{s,t}(k)$. It can be shown (as in the proof of \cite[Proposition 5.3]{Love2020}) that $D_s$ is finite for all $s\in k\setminus\{0\}$. 
		In a previous paper, the second author proved the following.
		\begin{prop}[{\cite[Corollary 1.4]{Love2020}}]\label{love}
			Let $s\in k$. For all $t_1,t_2\in k\setminus D_s$ such that $E_{s,t_1}(k)$ and $E_{s,t_2}(k)$ have rank $1$, $F^2(E_{s,t_1}\times E_{s,t_2})\new$ is finite. \black 
		\end{prop}

		Together with \autoref{corollaryS}, \autoref{rank0}, and \autoref{Sisogclasses}, this result can be used to create large sets of elliptic curves for which any product $X$ of curves from the set has finite $F^2(X)\new$.
		
		\begin{cor}
			Fix $s\in k$, and set 
			\[S_s=\{E_{s,t}:t\in k\setminus D_s,\,\rk E_{s,t}(k)=1\}.\]
			Suppose $C_1,\ldots,C_d$ are elliptic curves over $k$ such that for each $i$, either $\rk C_i(k)=0$ or $C_i$ is isogenous to a curve in $S_s$. Then $F^2(C_1\times\cdots C_d)\new$ is finite.
		\end{cor}

		As an example, setting $s=1$ and searching through elliptic curves of height at most $60^6$, the second author computed $27062$ pairwise non-isomorphic rank $1$ elliptic curves in $S_1$. More generally, for any fixed value of $s\in (k\setminus k^2)\cup\{1\}$, the one-parameter family $E_{s,t}$ (considered as an elliptic curve over $k(t)$) has generic rank $1$, with $P_{s,t}$ a point of infinite order~\cite[Proposition 5.3]{Love2020}. It is believed that any such family should have infinitely many specializations of rank $1$. If this is true, then for any $s\in (k\setminus k^2)\cup\{1\}$, the set $S_s$ will contain an infinite set of rank $1$ curves, pairwise non-isogenous over $\Q$, all of whose pairwise products have finite $F^2(X)\new$.
		
		
		We mention a few words about the construction. Let $E_1$ and $E_2$ be rank $1$ elliptic curves over $k$, and consider the degree $2$ map $\pi:E_1\times E_2\to Y$ obtained by taking the quotient by negation. The family $E_{s,t}$ is defined in order to guarantee existence of a rational curve in $Y$ passing through both $\pi(P_{s,t_1},P_{s,t_2})$ and $\pi(O_1,O_2)$.
		
		The paper~\cite{Love2020} then proceeds to take a principal divisor on this rational curve and pull it back along $\pi$ to produce the desired relation in $F^2(E_1\times E_2)$. We will provide another explanation for this relation. Note that any rational curve $\mathbb{P}^1\to Y$ defines a hyperelliptic curve\footnote{The rational curves defined in~\cite{Love2020} correspond to hyperelliptic curves of genus $5$.} $H$ mapping to $E_1\times E_2$ by taking the fiber product:
		\[\begin{tikzcd}
			H\arrow{r}{f}\arrow{d} &E_1\times E_2\arrow{d}{\pi}\\
			\mathbb{P}^1\arrow{r} &Y,
		\end{tikzcd}\]
		with $f(\overline{P})=-f(P)$ for all $P\in H(\overline{k})$. By composing with projection maps, we find maps $\phi_1:H\to E_1$ and $\phi_2:H\to E_2$ satisfying $\phi_i(\overline{P})=-\phi_i(P)$; in particular, Weierstrass points of $H$ must map to $2$-torsion points on each $E_i$. If $\pi(P_1,P_2)$ is in the image of the rational curve $\mathbb{P}^1\to Y$, then $(P_1,P_2)$ is in the image of $f$, so \autoref{torsionfromhyper} can be used to verify that $\{P_1,P_2\}_{k/k}$ is torsion.
		
		\begin{rem}\label{bielliptic}
			All examples discussed so far can be accounted for using Weil reciprocity on an elliptic or hyperelliptic curve, but there are examples using other curves.
			In his thesis~\cite{Lovethesis}, the second author provides some rank $1$ pairs $E_1,E_2$ with $F^2(E_1\times E_2)\new$ finite, not accounted for using any of the above techniques, for which the desired relation in $F^2(E_1\times E_2)$ comes from Weil reciprocity on a bielliptic curve. Using the Cremona labelling~\cite{cremona}, these examples include (37a1, 43a1), (37a1, 57a1), (37a1, 77a1), (53a1, 58a1), (61a1, 65a1), (61a1, 65a2), (65a2, 79a1) (\cite[Section 3.3.1]{Lovethesis}), and (37a1, 53a1) (\cite[Proposition 3.4.1]{Lovethesis}).
		\end{rem}
		
		\vspace{2pt}

		\vspace{30pt}
		
		\bibliographystyle{amsalpha}
		
		\bibliography{bibfile}
		
	\end{document}